\documentclass[a4paper,reqno]{amsart}
\usepackage[T1]{fontenc}
\usepackage[utf8]{inputenc}
\usepackage[english]{babel}
\usepackage{amsmath,amsfonts,amssymb}
\usepackage{graphicx}
\usepackage{hyperref}
\usepackage{enumerate}
\usepackage{xcolor}

\usepackage{amsthm}
\theoremstyle{definition}
\newtheorem{D}{Definition}[section]

\theoremstyle{plain}
\newtheorem*{theorem*}{Theorem}
\newtheorem{theorem}{Theorem}[section]
\newtheorem{lemma}[theorem]{Lemma}
\newtheorem{cor}[theorem]{Corollary}
\newtheorem{prop}[theorem]{Proposition}

\theoremstyle{definition}
\newtheorem{rem}[theorem]{Remark}

\newcommand{\R}{\ensuremath{\mathbb R}}

\newcommand{\s}{\ensuremath{\mathbb S}}
\newcommand{\eps}{\ensuremath{\varepsilon}}
\newcommand{\p}{\ensuremath{\frac{2n}{n-2}}}

\newcommand{\tub}[1]{\ensuremath{\Sigma^{#1}}}

\DeclareMathOperator{\vol}{Vol}

\title{Non-existence of Yamabe minimizers \\ on singular spheres}
\author{Kazuo Akutagawa${}^*$}
\address{Department of Mathematics, Chuo University, Tokyo 112-8551, Japan}
\email{akutagawa@math.chuo-u.ac.jp}

\author{Ilaria Mondello}
\address{Universit\'e de Paris Est Cr\'eteil, Laboratoire d'Analyse et math\'ematiques appliqu\'ees}
\email{ilaria.mondello@u-pec.fr}

\thanks{${}^*$\ 
supported in part by the Grants-in-Aid for Scientific Research (B), 
Japan Society for the Promotion of Science, No.~18HO1117.} 

\date{September, 2019.}
\begin{document}

\nocite{*}
\maketitle
\markboth{Non-existence of Yamabe minimizers on singular spheres} 
{Kazuo Akutagawa and Ilaria Mondello}

\begin{abstract}
We prove that a minimizer of the Yamabe functional does not exist for a sphere $\s^n$ of dimension $n\geq 3$, 
endowed with a standard edge-cone spherical metric of cone angle greater than or equal to $4\pi$, along a great circle of codimension two. 
When the cone angle along the singularity is smaller than $2\pi$, 
the corresponding metric is known to be a Yamabe metric, and we show that all Yamabe metrics in its conformal class are obtained from it by constant multiples and conformal diffeomorphisms preserving the singular set. 

\end{abstract}

\section*{Introduction}  

The Yamabe problem on a Riemannian smooth manifold $(M^n,g)$ with $n \geq 3$ consists in finding a metric with constant scalar curvature within the conformal class of a fixed metric. It is well known that one way of solving such problem is to minimize the Einstein-Hilbert functional over the conformal class of $g$, that is
$$ \mathcal{E}(\tilde{g})=\displaystyle \frac{\displaystyle \int_M R_{\tilde g}d\mu_{\tilde g}}{\mbox{Vol}_{\tilde g}(M)^{\frac{n-2}{n}}},  \quad \tilde{g}\in [g]= \{e^{2f}g, f \in C^{\infty}(M)\},$$
where $R_{\tilde{g}}, d\mu_{\tilde{g}}$ and ${\rm Vol}_{\tilde{g}}(M)$ denote respectively the scalar curvature, the volume form associated to $\tilde g$ and the volume of $(M, \tilde{g})$. The infimum of $\mathcal{E}$ over the conformal class $[g]$ is called the \emph{Yamabe constant} of $(M,g)$ and denoted by $Y(M,[g])$. We refer to a metric minimizing $\mathcal{E}$ as a \emph{Yamabe metric}. Thanks to the combined work of Yamabe, Trudinger, Aubin and Schoen (cf.\,\cite{LP87}, \cite{Au98}), we know that a Yamabe metric always exists on a compact smooth manifold without boundary. One of the key point in order to prove such a result is the \emph{Aubin's inequality}: for any compact smooth manifold $(M^n,g)$ we have
$$Y(M,[g]) \leq \bf Y_n$$
where $\bf Y_n$ denotes the Yamabe constant of the unit sphere $\s^n$ endowed with the round metric. Whenever the inequality is strict, there exists a Yamabe metric; in the case of equality, the manifold is conformally equivalent to the round sphere, and thus a Yamabe metric exists as well.  

It is natural to ask whether it possible to solve the Yamabe problem in other settings. For example, one can consider manifolds with conical or edge-cone singularities. For these latter, the local geometry at codimension two singularities is modeled on the product $\R^{n-2}\times C(\s^1_a)$, where $\s^1_a$ is the circle of length $\alpha=2\pi a$, $a>0$. 
We refer to $\alpha$ as the cone angle at the singularity. Edge-cone singularities naturally appear in the study of converging sequences of smooth manifolds: for instance, a recent result of  \cite{CJN} showed that when considering the pointed Gromov-Hausdorff limit $(X,d,p)$ of a non-collapsing sequence of smooth manifolds with a lower Ricci bound, the tangent cone at codimension two singularities is almost everywhere isometric to $\R^{n-2}\times C(\s^1_a)$, with $a \in (0,1]$. Metrics with edge-cone singularities have also been studied in various different contexts: for example, \cite{AL13} is concerned with obstructions to the existence of \emph{Einstein} edge-cone metrics on manifolds of dimension $4$; \cite{BV} is devoted to the study of the Yamabe flow on manifolds with edge-cone singularities. 
Furthermore, metrics with edge-cone singularities along a smooth divisor of real dimension two also played an important role 
in proving the existence of K\"ahler-Einstein metrics on Fano manifolds (see \cite{CDS15}, \cite{JMR}). 
It is then interesting to investigate to which extent the Yamabe problem on manifolds with edge-cone singularities corresponds to its smooth counterpart.

Thanks to \cite{ACM14}, an existence result for the Yamabe problem in the singular setting is available in the more general context of \emph{pseudomanifolds} and \emph{compact stratified spaces} carrying an \emph{iterated edge metric} (see \cite{ACM18} for more general settings), 
which include manifolds with conical singularities \cite{AB03}, orbifolds \cite{Ak12} and manifolds with edge-cone singularities \cite{M15}. 
The proof of such result relies on a generalized Aubin's inequality introducing a new conformal invariant, the \emph{local Yamabe constant}, which takes in account the local geometry at singular points. If the Yamabe constant of the whole space
is strictly smaller than the local Yamabe constant, then a Yamabe metric does exists (see the next section for definitions). Nevertheless, the situation is quite different from the smooth setting. First of all, the local Yamabe constant may be strictly smaller than $\bf Y_n$ and its value is in general unknown. A first result in \cite{Ak12} gave the local Yamabe constant in the case of orbifolds with isolated singularities; \cite{M17} provides an expression for the local Yamabe constant for a compact stratified space whose singularities have cross sections carrying an Einstein iterated edge metric. This applies in particular to general orbifolds and to the case of codimension two edge-cone singularities: 
if the cone angle $\alpha = 2\pi a$ along the singularity is smaller than $2\pi$, the local Yamabe constant is equal to $a^{\frac 2n}\bf Y_n$; it coincides with $\bf Y_n$ otherwise. 

Another remarkable difference with respect to the smooth setting, is that in in the case of equality between the local and the global Yamabe constants, examples of non-existence for the Yamabe minimizers do occur. Indeed, J.~Viaclovsky \cite{V10} gave a family of examples, each with one isolated orbifold singularity, for which the local and global Yamabe constants coincide and moreover a Yamabe minimizer does not exist. Such singular manifolds are constructed as the conformal orbifold compactification of any hyperk\"ahler ALE $4$-manifold with group of order $n >1$. 

We observe that in Viaclovsky's examples, the conical singularity has the maximal codimension in the ambient manifold. In this work, we are concerned in getting a better insight of the Yamabe problem for edge-cone singularities of the minimal codimension one can allow without creating a boundary, that is, codimension two.  As we stated above, in this case the value of the local Yamabe constant is known and it only depends on the dimension and on the cone angle at the singularities. We then consider a class of examples whose Yamabe constant coincide with that value. Those are given by spheres $\s^n$ endowed with metrics $h_a$ carrying an edge-cone singularity of angle $2\pi a$ along a great circle of codimension 2. Such metrics are in addition Einstein away from the singularity. 

When the cone angle is smaller than $2\pi$, we know that $h_a$ is a Yamabe metric (see  Corollary 3.2 in \cite{M17}). We additionally show that the conformal class of $h_a$ behaves precisely as the one of the smooth round metric: any other Yamabe metric in the conformal class is obtained from $h_a$ by a conformal diffeomorphism preserving the singular set, up to constant multiples.
 
The case of angle greater than $2\pi$ is substantially different. Indeed, $h_a$ is no longer a Yamabe metric. Moreover, our main result states that, when the cone angle is greater than or equal to $4\pi$, the conformal class of the singular metric does not admit any Yamabe metric. More precisely:

\quad \\ 
{\bf Main Theorem.}\ \ 
{\it Let $\s^n_a = (\mathbb{S}^n, h_a)$ be the $n$-dimensional sphere endowed with a standard edge-cone spherical metric of cone angle $\alpha = 2\pi a$ 
along a great circle $\mathbb{S}^{n-2}$ for some $a \geq 2$. 
Then, there is no solution to the Yamabe problem on $(\mathbb{S}^n, [h_a])$.
}\\ 

Our proof by contradiction depends on a combination of a branched covering version of Aubin's lemma for finite coverings \cite{Au76}, \cite{AN07}, 
the computation of the Yamabe constant of  $(\mathbb{S}^n, [h_a])$ \cite{M15}, \cite{M17}, and regularity results for Yamabe minimizers \cite{ACM15}, \cite{M15}, \cite{M18}. We conjecture that an analog result should hold true for $a \in (1, 2)$ as well. 

We point out that one of the reasons for investigating the Yamabe problem in a singular setting is an application 
to the study of the \emph{Yamabe invariant} \cite{Ko87} (or $\sigma$-invariant \cite{S89}) of a smooth compact manifold $M$.
This latter is a differential-topological invariant defined as the supremum of Yamabe constants over all conformal classes 
$$Y(M)=\sup_{g\in \mathcal{M}(M)} Y(M,[g])=\sup_{ g\in \mathcal{M}(M)}\inf_{\tilde g \in [g]}\mathcal{E}(\tilde g),$$
where $\mathcal{M}(M)$ denotes the set of smooth Riemannian metrics on $M$. It is a very difficult problem to explicitly compute, or even give estimates for the Yamabe invariant, especially when it is positive (see for example Section 1.2 in \cite{ADH}, or \cite{Mac}, for reviews of the known results about the Yamabe invariant). In an upcoming work \cite{AM}, we address to the question of obtaining a positive lower bound for the Yamabe invariant, given by the Yamabe constant of a singular Einstein metric with a edge-cone singularity of codimension two and cone angle smaller than $2\pi$. This can be done because, thanks to \cite{M17}, such metrics are known to be Yamabe metrics for angle smaller than $2\pi$. When the angle is greater than $2\pi$, our main theorem shows that the situation is more involved, because not only there exist singular Einstein metrics on the sphere that are not Yamabe metrics, but also their conformal class does not contain \emph{any} Yamabe metric. 

This paper is organized as follows. 
In the first section, we give some basic definitions and preliminary results about the singular spaces that we deal with, the metrics we consider and their local and global Yamabe constants. We define in particular singular spheres with edge-cone singularities and show that their local and global Yamabe constants coincide. 
Section 2 is devoted to prove regularity results about Schr\"odinger equations on stratified spaces endowed with an Einstein metric. Even if we use such results only in the case of singular spheres, we state them in their full generality, as they can be of interest in the study of Einstein edge-cone metrics in a more general setting.
In the two following sections, we focus on the edge-cone spheres $\s^n_a = (\s^n, h_a)$ with a cone angle $\alpha=2\pi a$ along a great circle of codimension two. 
In section 3, we consider the case of angle smaller than $2\pi$. Section 4 contains the proof of the main result of this paper.
The last section rephrases the previous results in terms of a Sobolev inequality in $\R^n \times C(\s^1)$. 
\quad \\ 
\quad \\ 
\noindent
{\bf Acknowledgements.} 
The first author would like to thank Jeff Viaclovsky for valuable discussions on the article Atiyah-LeBrun \cite{AL13}. \\

\section{Preliminaries}

We briefly recall some notions about smoothly stratified spaces, and refer to \cite{ACM14} and \cite{M17} for the precise definitions. 

A compact {\it stratified space} is a compact topological space $X$ that can be decomposed into two disjoint subsets, 
the {\it regular set} $\Omega$, which is a smooth open manifold of dimension $n$, dense in the whole space, 
and the {\it singular set} $\Sigma$. 
This latter has different components $\Sigma^j$, the {\it singular strata}, 
that are smooth manifolds of dimension $j$, with $0\leq j \leq (n-2)$. 
Assume that each singular stratum $\Sigma^j$ has only finitely many connected components. 
Each point $x$ of a singular stratum $\Sigma^j$ has a neighbourhood which is homeomorphic to the product $\mathbb{B}^j(\eps) \times C_{[0,\eps)}(Z_x)$, where $\mathbb{B}^j(\eps)$ is an Euclidean open ball in $\R^j$ and $ C_{[0,\eps)}(Z_x)$ is the truncated cone over $Z_x$, a compact stratified space which is called the {\it link} of the stratum $\Sigma^j$. 

An {\it iterated edge metric} on a stratified space can be defined by iteration on the dimension; 
it is a Riemannian metric on the regular set with precise asymptotics close to the singular stratum 
(see for example Section 3 in \cite{ALMP}). 
Here, we are mostly interested in the case of one singular stratum of minimum codimension two. 
In this case, the link at any point is a circle and we can define a model metric on $\R^{n-2}\times C(\s^1)$ 
$$h + dr^2 + a^2r^2d\theta^2,$$
where $h$ is a Riemannian metric on $\R^{n-2}$ and the parameter $a > 0$ determines the angle $\alpha = 2\pi a$ of the two-dimensional cone $C(\s^1)$. 
For a stratified space with only one stratum of codimension two, an iterated edge metric, also called an edge-cone metric (cf.\,\cite{AL13}), is defined as follows: 

\begin{D}
Let $X$ be a compact stratified space of dimension $n$ with one singular stratum of codimension 2, denoted $\Sigma^{n-2}$. 
An {\it edge-cone metric} $g$ on $X$ is a Riemannian metric on the regular set $\Omega = X - \Sigma^{n-2}$ 
such that there exist positive constants $\gamma$, $\Lambda$ and for any $x \in \Sigma^{n-2}$ there exists $a_x > 0$ such that 
$$
|g-(h+dr^2+a_x^2r^2d\theta^2)| \leq \Lambda r^{\gamma}.
$$
Here, $r = r(\cdot) = d(\cdot\,, \Sigma^{n-2})$. 
We refer to $\alpha_x = 2\pi a_x$ as the {\it cone angle} of $\Sigma^{n-2}$ at the point $x$.
\end{D}

All the curvatures that are associated with a Riemannian metric $g$, scalar curvature $R_g$ and Ricci curvature ${\rm Ric}_g$, are well-defined on the regular set $\Omega$. 
Therefore, an Einstein metric on a compact stratified space is defined by an iterated edge metric $g$ 
such that there exists $\lambda \in \R$ for which ${\rm Ric}_g = \lambda g$ on $\Omega$. 

\subsection*{The Yamabe Problem on stratified spaces} We follow here the lines of \cite{ACM14}. 
Let $X$ be a compact stratified space of dimension $n\geq 3$. 
We can define the Sobolev space $W^{1,2}(X)$ as the $W^{1,2}$-completion of Lipschitz functions on X with the usual Sobolev norm, 
and the Laplacian as the Friedrichs operator associated to the Dirichlet energy. 
Given an iterated edge metric $g$ such that the scalar curvature $R_g$ belongs to $L^q(X)$ for $q >\frac n2$, 
the Yamabe functional is defined as follows: for a function $u$ in $W^{1,2}(X)$ 
\begin{equation*}
Q_g(u) := \frac{E_g(u)}{||u||_{\p}^2}=\frac{\displaystyle \int_X (a_n|du|^2 + R_gu^2)d\mu_g}{||u||_{\p}^2}. 
\end{equation*}
Similar to the smooth case, the Yamabe constant $Y(X, d)$ is defined as the infimum of $Q_g$: 
\begin{equation*}
Y(X, d) := Y(\Omega, [g]) = \inf \{ Q_g(u) |\ u \in C^1_0(\Omega),\ u \not\equiv 0 \} 
\end{equation*} 
\begin{equation*}
\qquad \qquad \qquad \qquad \qquad \ \ = \inf \{ Q_g(u) |\ u \in W^{1,2}(X),\ u \not\equiv 0 \}. 
\end{equation*}
As in the smooth setting, \emph{Aubin's inequality} holds, that is, the Yamabe constant $Y(X, [g])$ is always smaller than or equal to ${\bf Y}_n$. 

A {\it Yamabe minimizer} is a positive function $u$ such that $Q_g(u)$ equals the Yamabe constant.
If such $u$ exists, the metric $\widetilde{g}=u^{\frac{4}{n-2}}g$ is called a {\it Yamabe metric} and it has constant scalar curvature on the regular set. 
As shown by Theorem 1.12 in \cite{ACM14}, if a Yamabe minimizer does exist, it belongs to $W^{1,2}(X)\cap L^{\infty}(X)$ and it satisfies the Yamabe equation on the regular set $\Omega$: 
\begin{equation}
\label{Yeq}
- a_n \Delta_g u + R_g u = Y(M,[g])u^{\frac{n+2}{n-2}}, 
\end{equation}
where $- \Delta_g$ denotes the non-negative Laplacian of $g$.

For an open ball $B_r(p)$ (of radius $r$ centered at $p$ ), 
we define the Yamabe constant of the ball similarly, 
by taking the infimum of $Q_g$ over $W^{1,2}_0(B_r(p)\cap \Omega)$: 
\begin{equation*}
Y(B_r(p)) =\inf \{ Q_g(u)\ |\ u \in W^{1,2}_0(B_r(p)\cap \Omega),\ u \not\equiv 0 \}, 
\end{equation*}
Here, $W^{1,2}_0(B_r(p)\cap \Omega)$ denotes the $W^{1,2}$-completion of $C^1_0(B_r(p)\cap \Omega)$. 
This allows us to introduce the {\it local Yamabe constant} 
\begin{equation*}
Y_{\ell}(X) =\inf_{p \in X}\lim_{r \rightarrow 0}Y(B_r(p)). 
\end{equation*}
As in the smooth case, a generalized Aubin's inequality holds 
$$ 
Y(X, d) \leq Y_{\ell}(X). 
$$
The value of the local Yamabe constant is known in the case of orbifolds \cite{Ak12} or more generally when the links of the singular strata carry Einstein metrics \cite{M17}. This latter result holds in particular when there is only one stratum of codimension two. As stated in the introduction, in this case it is indeed possible to give an expression of the local Yamabe constant that only depends on the dimension and on the angle along the singular stratum. More precisely: 

\begin{theorem}{\emph{(\cite{M17})}}
\label{localYamabe}
Let $(X^n,g)$ be a compact stratified space with singular set $\Sigma^{n-2}$ of codimension two, 
endowed with an edge-cone metric. 
One of the two possibilities holds:
\begin{itemize}
\item[(i)] If there exists $x \in \Sigma^{n-2}$ such that $\alpha_x \leq 2\pi$, then
$$
Y_{\ell}(X) = \Big{(} \frac{\alpha}{2\pi}\Big{)}^{\frac{2}{n}} {\bf Y}_n,\quad \alpha := \inf_{x \in \Sigma^{n-2}}\alpha_x. 
$$  
\item[(ii)] If for all $x \in \Sigma^{n-2}$, we have $\alpha_x > 2\pi$, then $Y_{\ell}(X) = {\bf Y}_n$. 
\end{itemize}
\end{theorem}

Thanks to Theorem 4.1 in \cite{M18}, we also know that if $\alpha_x \leq 2\pi$ for all $x \in \Sigma^{n-2}$, then an Einstein metric is also a Yamabe metric. 
This is not necessarily true when the cone angles along the singular set are greater than $2\pi$ (see Remark \ref{EinsteinNotYamabe} for a simple example). 

\subsection*{Standard edge-cone spheres} 
We are now going to define an edge-cone metric on the unit sphere $\s^n$ in such a way that it has a cone angle along a great circle of codimension 2. For $a\in \R_{> 0}$, we consider the following double warped product metric: 
$$ h_a =d\rho^2+\sin^2\rho\,g_{\s^{n-2}}(x) +a^2\cos^2\rho\,d\theta^2\ \mbox{ on } (\rho, x, \theta) \in X=\left(0, \frac{\pi}{2}\right)\times \s^{n-2}\times \s^1,$$ 
where $g_{\s^{n-2}}$ denotes the standard round metric of constant curvature $1$ on $\s^{n-2}$. The {\it edge-cone sphere} $\s^n_a$ is defined to be the 
metric completion of $X$ with respect to the metric defined by $h_a$. We refer to $h_a$ as the standard {\it edge-cone spherical metric} (cf.\,\cite{AL13}). 
When $a$ equals 1, this is simply the round $n$-sphere $\s^n$. 
Otherwise, $\s^n_a = (\s^n, h_a)$ is a compact stratified space, homeomorphic to the sphere, with one singular stratum of codimension two given by: 
$$\Sigma=\Big{(} \left\{\frac{\pi}{2}\right\} \times \s^{n-2}\times \s^1 \Big{)}\Big{/}\sim\ = \s^{n-2} \times \{\ast\}.$$ 
Here, $\ast$ denotes the conical singular point in the spherical cone $( \left(0, \frac{\pi}{2}\right)\times \s^1, d\rho^2 + a^2\cos^2\rho\,d\theta^2 )$. 
Observe that the metric $h_a$ on the regular set is of constant curvature $1$, and hence it is an Einstein metric with ${\rm Ric}_{h_a} = (n-1)\,h_a$. 

Now consider $\R^n$ seen as the product $\R^{n-2}\times \R^2 (\ni (x, r))$ and define the metric:
\begin{equation*}
g_{a}= g_{\mathbb{E}^{n-2}}(x) + dr^2 + a^2r^2d\theta^2,
\end{equation*}
where $g_{\mathbb{E}^{n-2}}$ denotes the Euclidean metric on $\R^{n-2}$. 
As above, when $a$ equals $1$, the previous metric is the Euclidean metric; 
otherwise we have a metric with an edge-cone singularity of cone angle $\alpha =2\pi a$ along $\R^{n-2}$.

The metrics $h_a$ on $\s^n$ and $g_a$ on $\R^n$ are conformally equivalent via the stereographic projection, as in the smooth case. 
Indeed, let $N$ be a point in the singular set $\Sigma$ of $\s^n_a$ 
and denote by $\pi$ be the stereographic projection from $\s^n - \{N\}$ onto $\R^n$. 
If $\Phi$ is its inverse, it is easy to check that the pull-back of $h_a$ is conformal to $g_a$ and it is equal to: 
\begin{equation*}
\Phi^*h_a = \left(\frac{2}{1+|x|^2+r^2}\right)^2g_a,
\end{equation*}
where $|\cdot|$ denotes the Euclidean norm in $\R^{n-2}$. 
We also observe that $(\s^n, h_a)$ can be seen as the spherical suspension of $(\s^{n-1}, h_a^{n-1})$, meaning that it is isometric to:
$$([0,\pi]\times \s^{n-1}, dt^2+\sin^2t\,h_a^{n-1}).$$

For the sake of simplicity, from now on we denote the local Yamabe constant of $(\s^n, [h_a])$ by $Y_a$, 
its Yamabe functional $Q_{h_a}$ by $Q_a$, the energy $E_{h_a}$ by $E_a$, the Laplacian and volume form associated to $h_a$ by $\Delta_a$ and $d\mu_a$ respectively. 

Since $h_a$ is an Einstein metric on $\Omega$, we can apply Theorem \ref{localYamabe} and get the value of the local Yamabe constant of $(\s^n, [h_a])$:
\begin{cor}
Let $a \in \R_{>0}$ and consider $\s^n_a = (\s^n, h_a)$ defined as above. Then its local Yamabe constant $Y_a$ satisfies:
\begin{itemize}
\item[(a)] if $a \in (0,1)$, $Y_a=a^{\frac{2}{n}}{\bf Y}_n$, 
\item[(b)] if $a \geq 1$, $Y_a={\bf Y}_n$. 
\end{itemize}
\end{cor}
Moreover we have the following result:
\begin{lemma}
For any $a\in \R_{> 0}$, the local Yamabe constant and the Yamabe constant of $(\s^n, [h_a])$ coincide: 
$$Y(\s^n, [h_a])=Y_a.$$
\end{lemma}
\begin{proof}
A simple computation shows that for any $a$:
$$Q_a(h_a)=a^{\frac{2}{n}}{\bf Y}_n$$
See for instance Lemma 4.3 in \cite{M15}. If $a < 1$, $h_a$ is a Yamabe metric and thus the Yamabe constant $Y(\s^n, [h_a])$ coincides with $Q_a(h_a)$. Together with statement (a) in the previous Corollary, this directly shows that $Y(\s^n, [h_a])=Y_a$.
 
If $a \geq 1$, then $(\s^n, [h_a])$ is conformally equivalent to $(\R^n, [g_a])$, 
and thanks to Proposition 4.7 in \cite{M17} we know that 
$$Y(\s^n, [h_a]) = Y(\R^n,[g_a])={\bf Y}_n = Y_a.$$
\end{proof}
\begin{rem}
\label{EinsteinNotYamabe}
The fact that $Q_a(h_a)$ equals $a^{\frac{2}{n}}{\bf Y}_n$ does not depend on $a$ being smaller than $1$. 
In particular, the previous lemma shows that the Einstein metric $h_a$ is \emph{not} a Yamabe metric when $a > 1$. 
\end{rem}
%

\section{Regularity results}
Our proofs rely on having precise information about the regularity of a Yamabe minimizer and its gradient. Such regularity is not specific to the particular case of edge-cone spheres, nor to the only presence of codimension two singularities, but it actually holds on a general compact stratified space, with possibly higher codimension strata whose links are not necessarily smooth. The results of this section are therefore stated in their full generality.


We are going to prove the following:

\begin{prop}
\label{reg}
Let $(X^n,g)$ be a compact stratified space with singular set $\Sigma$, endowed with an Einstein edge-cone metric. 
Assume that there exists a Yamabe minimizer $u \in W^{1,2}(X)\cap L^{\infty}(X)$. Then the following hold: 
\begin{itemize}
\item[(i)] If for all $x \in \Sigma^{n-2}$ the cone angle at $x$ satisfies $\alpha_x\leq 2\pi$, then $u$ belongs to $W^{2,2}(X)$ and its gradient is bounded. In particular, $u$ is a Lipschitz function.
\item[(ii)] If there exists $x\in \Sigma^{n-2}$ such that the cone angle at $x$ is $\alpha_x>2\pi$, then $u$ belongs to $C^{0,\nu}(X)$, with $\nu= \nu(X) \in (0,1)$. 
Moreover for any $\eps >0$ we have 
$$||du||_{L^{\infty}(X \setminus \tub{\eps})}\leq C \eps^{\nu -1}.$$ 
\end{itemize}
In both cases, the integration by parts formula holds true for $u$  
\begin{equation}
\label{IPP}
- \int_X u\Delta_g ud\mu_g = \int_X |du|^2d\mu_g.
\end{equation}
\end{prop}


The first statement (i) was proven in \cite{M18}, Lemma 4.6. 
The integration by parts formula then follows easily by choosing a family of cut-off functions vanishing on an $\eps$-tubular neighborhood $\Sigma^{\eps}$ of the singular set 
and whose norm in $L^2$ tends to zero as $\eps$ goes to zero (see for example $f_{\eps}$ in the proof of Theorem 2.1 in \cite{M17}). 

The rest of this section is devoted to the proof of the second case (ii) of Proposition \ref{reg}. In order to do that, first observe that we can consider the Yamabe equation \eqref{Yeq} as a Schr\"odinger equation on the regular set $\Omega$ 
$$- \Delta_g u = Vu,$$
where 
$$V=a_n^{-1}(Y(X,[g]u^{\frac{4}{n-2}}-R_g)\in L^{\infty}(X).$$
Since $g$ is an Einstein metric, the scalar curvature $R_g$ is constant, 
therefore $V$ belongs to $L^{\infty}(X)$. We can also define the locally Lipschitz function 
$$F(x)=a_n^{-1}(Y(X,[g])x^{\frac{4}{n-2}}-R_g)x,$$ 
and the Yamabe equation can be written as an equation of the form 
$$- \Delta_g u=F(u).$$
The proof of Proposition \ref{reg} combines these two ways of considering the Yamabe equation. 

\subsection*{Regularity for solutions of Schr\"odinger equations} For a compact stratified space $(W^n,h)$ endowed with an iterated edge metric, 
we denote by $\lambda_1(W)$ the first non-zero eigenvalue of the Laplacian $- \Delta_h$ and, following \cite{ACM15}, we define 
\begin{equation*}
\nu_1(W)= \begin{cases}
1 & \mbox{ if } \lambda_1(W)\geq n \\
\mbox{the unique value in $(0,1)$ s.t. } \\
\nu_1(W)(n-1+\nu_1(W)) & \mbox{ if } \lambda_1(W) < n. 
\end{cases}
\end{equation*}
Then, for a compact stratified space $(X^n,g)$, we define 
\begin{equation*}
\nu(X)=\inf_{x \in X}\nu_1(Z_x), 
\end{equation*}
where $Z_x$ is the link at $x$. 

Consider for example a stratified space which only has codimension 2 singularities. Then the links are circles $\s^1_x=(\s^1, (\alpha_x/2 \pi)^2d\theta^2)$ of length $\alpha_x$. Therefore we have 
\begin{equation*}
\nu_1(\s^1_x)=\begin{cases}
1 & \mbox{ if } \alpha_x \leq 2\pi \\
\displaystyle\frac{1}{a_x} & \mbox{ if } \alpha_x =2\pi  a_x > 2\pi. 
\end{cases}
\end{equation*}
In this case we can then express the value $\nu(X)$ in terms of the cone angles $\alpha_x$: 
\begin{equation}
\label{nu}
\nu(X)=\begin{cases}
1 & \mbox{ if } \forall x\in \Sigma \quad  \alpha_x \leq 2\pi,\\
\displaystyle\inf_{x\in \Sigma} \frac{2\pi}{\alpha_x} &\mbox{ if } \exists x\in \Sigma \quad \alpha_x >2\pi. 
\end{cases}
\end{equation}
For the sake of simplicity, we assume that $\nu(X) >0$, that is, there isn't any sequence of singular points for which the angle tends to infinity.

The Lichnerowicz theorem for stratified spaces (Theorem 2.1 in \cite{M17}) ensures that expression \eqref{nu} for $\nu(X)$ holds true whenever the metric has Ricci tensor bounded from below on the regular set. Indeed we have the following:

\begin{lemma}
Let $(X^n,g)$ be a compact stratified space with singular set $\Sigma$. Assume that there exists $k \in \R$ such that $Ric_g\geq k\, g$ on the regular set $X - \Sigma$. 
Then, $\nu(X)$ is given by \eqref{nu}. 
\end{lemma}

\begin{proof}
The lower bound on the Ricci tensor on the regular set implies that for any $x \in \Sigma$ the tangent cone at $x$ is Ricci flat on its regular set, and in particular any link $Z_x=(Z, k_x)$ is such that $Ric_{k_x}\geq (\mbox{dim}(Z) -1)\, k_x$ on the regular set of $Z$ (see Lemma 2.1 in \cite{M15}). In particular, when $\mbox{dim}(Z)> 1$, that is, when $x$ belongs to a singular stratum of codimension greater than two, we can apply Lichnerowicz theorem to $(Z,k_x)$ and get $\lambda_1(Z_x)\geq \mbox{dim}(Z)$. As a consequence, whenever $x$ belongs to a singular stratum of codimension greater than two, $\nu(Z_x)=1$ and therefore $\nu(X)$ only depends on angles at the codimension 2 stratum. This proves \eqref{nu}. 
\end{proof}


By combining Theorem A in \cite{ACM15} and Moser's iteration technique (see Lemma 1.16 in \cite{M15}), 
we obtain the following result for weak solutions of a Schr\"odinger equation:
\begin{prop}
\label{ThmA}
Let $(X^n,g)$ be a compact stratified space endowed with an iterated edge metric and $V\in L^{q}(X)$ for $q > \frac{n}{2}$. 
Assume that $u \in W^{1,2}(X)\cap L^{\infty}(X)$ is a weak solution of 
\begin{equation}
\label{shrod}
- \Delta_g u =Vu
\end{equation}
and moreover there exists a constant $c>0$ such that 
\begin{equation*}
\label{grad}
\Delta_g |du| \geq - c |du| \mbox{ on } \Omega.
\end{equation*}
Then for any $\eps >0$ we have 
\begin{itemize}
\item[(i)] if $\nu(X) = 1$ and $V \in L^{\infty}(X)$, there exists a positive constant $C$ such that 
\begin{equation*}
||du||_{L^{\infty}(X \setminus \tub{\eps})}\leq C\sqrt{|\ln(\eps)|},
\end{equation*}
\item[(ii)] if $\nu(X) \in (0,1)$ and $V \in L^{\infty}(X)$, then $u \in C^{0,\nu}(X)$ for $\nu=\nu(X)$ and there exists a positive constant $C$ such that 
\begin{equation*}
||du||_{L^{\infty}(X \setminus \tub{\eps})}\leq C \eps^{\nu -1}, 
\end{equation*}
\item[(iii)] if $\nu(X) \in (0,1]$ and $q \in \left(\frac n2,\infty \right)$, then $u \in C^{0,\mu}(X)$ for 
$$\mu= \min\left\{ \nu(X), 1- \frac{n}{2q}\right\}, $$
and there exists a positive constant such that
\begin{equation*}
||du||_{L^{\infty}(X \setminus \tub{\eps})}\leq C \eps^{\mu -1}.
\end{equation*}
\end{itemize}
\end{prop}
In cases (ii) and (iii), we deduce the integration by parts formula \eqref{IPP} for a weak solution $u$ of \eqref{shrod}. 

\begin{lemma}
\label{LemmaIPP}
Let $(X^n,g)$ be a compact stratified space with singular set $\Sigma$, 
endowed with an iterated edge-cone metric.  
Assume that there exists $x$ in the stratum of codimension two such that $\alpha_x >2\pi$ and that $\nu(X) > 0$. 
For $u\in W^{1,2}(X)\cap L^{\infty}(X)$  satisfying the same assumptions of Proposition \ref{ThmA}, the integration by parts formula \eqref{IPP} holds. 
\end{lemma}

\begin{proof}
Fix $\eps >0$ and consider an $\eps$-tubular neighbourhood $\tub{\eps}$ of the singular set. Then we can write: 
\begin{equation}
\label{boundary}
\int_{X}|du|^2d\mu_g= \lim_{\eps \rightarrow 0}\left( - \int_{X \setminus \tub{\eps}} u\Delta_{g} u d\mu_g + \int_{\partial \tub{\eps}}u \langle du,N \rangle d\sigma_g \right),
\end{equation}
where $N$ is the unit outward normal of $\partial \tub{\eps}$. We are going to show that the second term in the previous limit tends to zero as $\eps$ goes to zero. In order to do that, observe that $u$ is bounded and the norm in $L^{\infty}$ of its gradient is controlled  by a constant times $\eps^{\beta}$ away from $\Sigma^{\eps}$. 
The exponent $\beta$ is either equal to $\nu(X)$ or to $\mu$; in both cases $0 <\beta <1$. By using this and Cauchy-Schwarz inequality we get the estimate 
\begin{equation*}
\int_{\partial\tub{\eps}} u \langle du,N \rangle d\sigma_g \leq ||u||_{\infty}\eps^{\beta-1}\vol_{\sigma_g}(\partial \tub{\eps})\leq c \eps^{\beta}, 
\end{equation*}
for some positive constant $c$. 
Here we used that since $\Sigma$ has minimal codimension 2, the volume of the boundary of an $\eps$-tubular neighbourhood is bounded by a constant times $\eps$.
 Since $\beta >0$, by taking the limit in \eqref{boundary} when $\eps$ goes to zero we obtain the desired equality \eqref{IPP}. 
\end{proof}

\subsection*{Proof of Proposition \ref{reg}} 
We observed at the beginning of this section that, in the Einstein case, a Yamabe minimizer $u$ is a weak solution of a Schr\"odinger equation for $V \in L^{ \infty}(X)$. Hence, in order to prove Proposition \ref{reg}, it suffices to show that the gradient of $u$ satisfies inequality 
$$\Delta_g |du |\geq - c|du|$$
on the regular set,  for some positive constant $c$. 

The following is proven in Proposition 2.3 in \cite{M15}, that we recall here: 
\begin{prop}
\label{regRicciBd}
Let $(X^n,g)$ be a compact stratified space endowed with an iterated edge metric such that for some $k \in \R$ the inequality ${\rm Ric}_g\geq k\,g$ holds on the regular set. Let $F$ be a locally Lipschitz function on $\R$ and $u$ a non-negative function in $W^{1,2}(X)\cap L^{\infty}(X)$ which is a weak solution of 
\begin{equation}
\label{locLip}
- \Delta_g u = F(u)
\end{equation}
Then there exists a positive constant $c$ such that $\Delta_g|du| \geq - c|du|$ on the regular set. 
\end{prop}
As we stated above, the Yamabe equation can be seen as an equation of the form \eqref{locLip}. 
As a consequence, the assumptions of Proposition \ref{reg} ensure that we can apply both Proposition \ref{regRicciBd} and Proposition \ref{ThmA}. We then get that a Yamabe minimizer belongs to the H\"older space $C^{0,\nu}(X)$ for $\nu=\nu(X)$ given by expression \eqref{nu}. Moreover, Lemma \ref{LemmaIPP} proves the integration by parts formula. 

This concludes the proof of Proposition \ref{reg}. 
\begin{rem}
If we denote by $L_g = - a_n\Delta_{g} + R_{g}$ the conformal Laplacian associate to the metric $g$, Proposition \ref{reg} implies in particular the following equality, which will be useful in the following
\begin{equation}
\label{confIPP}
\int_{X}uL_gud\mu_g=E_g(u)=\int_{X}(a_n|du|^2 + R_g u^2)d\mu_{g}.
\end{equation}
\end{rem}

\section{Singular spheres of angle smaller than $2\pi$}
We consider the standard edge-cone sphere $\s^n_a = (\s^n, h_a)$ for $a \leq 1$. 
In this case, $h_a$ is a Yamabe metric and we are going to show that the behavior of its conformal class is analog to the one of the round metric $h_1 = g_{\mathbb{S}}$. 
For the smooth round sphere $(\s^n,h_1)$, any Yamabe metric, not homothetic to $h_1$, is obtained from $h_1$ by a conformal diffeomorphism (see \cite{O71} or Proposition 3.1 in \cite{LP87}).

Now consider a Yamabe minimizer $u$ in the conformal class of $h_a$, $h=u^{\frac{4}{n-2}}h_a$. 
We can eventually apply Proposition \ref{reg} to $u$ and obtain that it is a Lipschitz function. 
This, together with Corollary 4.8 of \cite{M18} allows us to prove the following: 

\begin{prop}
Let $a \leq 1$. 
If there exists a conformal metric $h \in [h_a]$ with constant scalar curvature, 
then there exists a conformal diffeomorphism $\varphi$ of $(\s^n,h_a)$, 
preserving the singular set, such that $h=\varphi^*h_a$, up to a constant multiple . 
As a consequence, the Yamabe functional on $(\s^n, h_a)$ is minimized by constant multiples of $h_a$ and its images under conformal diffeomorphism. 
\end{prop}

\begin{proof} We already know that $h_a$ minimizes the Yamabe functional when $a \leq 1$. 
Therefore it suffices to show the existence of the desired conformal diffeomorphism. 
Let $h= u^{\frac{4}{n-2}}h_a$ be a conformal metric with constant scalar curvature. 
Up to a constant multiple, we may assume that $R_h=n(n-1)$. 
By Theorem 4.3 in \cite{M15}, the metric $h$ is an Einstein metric. 
Moreover, by Corollary 4.8 in \cite{M15}, 
$(\s^n, h)$ is isometric to a spherical suspension: 
that is, there exists a stratified space $X$ of dimension $(n-1)$, with singularities of dimension $(n-3)$ and angles smaller than $2\pi$, 
endowed with an Einstein metric $g$, and an isometry $\varphi$ from $(\s^n, h)$ to the warped product $(X\times [0,\pi], dt^2+\sin^2t\,g)$. 
The isometry $\varphi$ preserves the singular set, meaning that it sends $\Sigma^{n-2}$ onto $X^{\tiny{sing}}\times [0,\pi]$. 
We know that the the point $P= X \times \{0\}$ belongs to the singular set of $X\times [0,\pi]$ 
and by the proof of the Obata singular theorem in \cite{M18} the tangent cone at $P$ is a cone over $X$. 
More explicitly, the truncated tangent cone at $P$ is given by the following Gromov-Hausdorff limit: 
$$(C_{[0,1)}(X), ds^2+s^2g)=\lim_{\eps\rightarrow 0}(B_{\eps}(P), \eps^{-2}(dt^2+\sin^2t\,g), P),$$
Moreover, the convergence is in $C^{\infty}_{loc}$ on the regular sets. 
Now consider $x = \varphi^{-1}(P)\in \Sigma^{n-2}$. 
Since $\varphi$ is an isometry, the truncated tangent cone at $x$ with respect to the metric $h$ is isometric to $C_{[0,1)}(X)$: 
$$\lim_{\eps\rightarrow 0}(B_{\eps}(x), \eps^{-2}h, x)=(C_{[0,1)}(X), ds^2+s^2g).$$
We also know that the tangent cone at $x$ with respect to the metric $h_a$ is a cone over $\s^{n-1}_a$. 
Indeed, $h_a$ is isometric to the warped product metric $d\rho^2+\sin^2\rho\,h^{n-1}_a$ on $[0,\pi]\times \s^{n-1}$. 
On the regular part of the ball $B_{\eps}(x)$, we can consider the change of coordinates $\rho=\eps r$ and as $\eps$ goes to zero we obtain
$$\eps^{-2}h_a = \eps^{-2}(\eps^2dr^2+\sin^2(\eps r)h_a^{n-1}) \rightarrow dr^2+r^2h_a^{n-1}.$$
For the conformal metric $h=\phi h_a$, with $\phi=u^{\frac{4}{n-2}}$,
we know that $\phi$ is a Lipschitz and bounded function, thanks to Proposition \ref{reg} (i). Therefore, for any $z \in B_{\eps}(x)$, we have 
$$|\phi(z)-\phi(p)|\leq C \eps.$$
As a consequence, when considering the limit as $\eps$ goes to zero of the metric $\eps^{-2}h$, 
we have, with the same change of coordinates as above, on the regular part of $B_{\eps}(x)$: 
$$\eps^{-2}h=\eps^{-2}\phi h_a \rightarrow dr^2+r^2h_a^{n-1}.$$
But the tangent cone at $x$ is unique, and therefore $C(X^{\mbox{\tiny reg}})$ and $C(\s^{n-1,\mbox{\tiny reg}}_a)$ must be isometric. Moreover, the convergence is in the pointed Gromov-Hausdorff sense, then the isometry sends the vertex of $C(X^{\mbox{\tiny reg}})$ to the vertex of $C(\s^{n-1,\mbox{\tiny reg }}_a)$. Now, both $s$ and $r$ are the distances from the vertices of the cone, and therefore each slice $\{s\}\times X^{\mbox{\tiny reg}}$ is isometric to a slice $\{r\}\times \s^{n-1,\mbox{\tiny reg}}_a$. By taking the metric completions of the regular set, we deduce that there is an isometry between $X$ and $\s^{n-1}_a$ preserving the singular sets. As an immediate consequence $(\s^n, h)$ is isometric to $(\s^n, h_a)$. More precisely, we have shown that there exists an isometry $\varphi : (\s^n, h)\rightarrow (\s^n, h_a)$ which preserves the singular setting. Since the pullback $\varphi^*h_a$ is conformal to $h_a$, $\varphi$ is the conformal diffeomorphism of the sphere $(\s^n, h_a)$ that we were looking for. Therefore, for any conformal metric $h$ of constant scalar curvature there exists a conformal diffeomorphism of $(\s^n, h_a)$, preserving the singular set, such that $h =\varphi^*h_a$, as we wished. \end{proof}


This result allows us to express the conformal factor of any metric $h$ minimizing the Yamabe functional on $(\s^n,h_a)$ 
in terms of functions defined in $(\R^{n-2}\times C(\s^1),g_a)$, in the same way as in the smooth case. 
Indeed, consider $N \in \Sigma$ and the stereographic projection from $\s^n - \{N\}$ to $\R^n$. 
Denote by $\Phi$ its inverse; as we noticed above, the stereographic projection is as usual a conformal diffeomorphism such that 
$$\Phi^*h_a = 4u^{\frac{4}{n-2}}g_a,$$
where we denote 
$$u(x,r)=(1+|x|^2+r^2)^{\frac{2-n}{2}}.$$
Consider the conformal diffeomorphism $\varphi$ constructed in the previous Proposition and define 
$$ \psi: \R^n \rightarrow \R^n, \quad \psi = \pi \circ \varphi \circ \Phi.$$
Then $\psi$ is a conformal diffeomorphism of $(\R^n,g_a)$ preserving the singular set $\mathcal{S}=\R^{n-2}\times \{0\}$. Now, any conformal diffeomorphism of the Euclidean space preserves the singular set, so $\psi$ can be expressed as the composition of a rotation, a translation $\tau_v$ in the direction $v$ and a dilation $\delta_{\lambda}(x)=\lambda x$. 
It is then possible to write, as in the smooth case: 
\begin{equation*}
\Phi^*\psi^*h_a = 4u_{\lambda,v}^{\frac{4}{n-2}}g_a, \end{equation*}
where $u_{\lambda,v}=u\circ\tau_v\circ\delta_{\lambda}$. This gives all Yamabe minimizers on the edge-cone sphere $\s^n_a = (\s^n, h_a)$.


\section{Singular spheres of cone angle greater than or equal to $4\pi$ and non-existence}
This section is devoted to proving that a non trivial Yamabe minimizer does not exists on standard edge-cone spheres $\s^n_a = (\s^n, h_a)$ with $a \geq 2$. 
The proof is by contradiction and it relies on Theorem \ref{localYamabe} and on a lemma by Aubin \cite{Au76} for the Yamabe constant of Riemannian normal coverings (see \cite{AN07} for the case of non-normal coverings). 

{\it Aubin's lemma} states that if $(M^n,g)$ is a smooth compact manifold with positive Yamabe constant, 
and $M_k$ a finite covering of order $k > 1$ of $M$, then the Yamabe constants of $M_k$ is strictly larger than the one of $M$. 
Now, if we consider $a \geq 2$ and $b=a/2$, $\s^n_a = (\s^n, h_a)$ is a double branched cover of $\s^n_b = (\s^n, h_b)$.
 If we assume the existence of a Yamabe minimizer,
 the regularity and integration by parts formula proven in Proposition \ref{reg} will allow us to show Aubin's lemma in the setting of singular spheres. As a consequence, we will get 
$$Y(\s^n, [h_a])>Y(\s^n, [h_b]).$$
But we know that for $a,b \geq 1$ both of the Yamabe constant coincide with the one of the round sphere. 
Therefore, we will obtain a contradiction and a Yamabe minimizer cannot exist. 

In the following, we give the details of the proof of the Main Theorem.

\begin{theorem}
Let $(\s^n, h_a)$ be the standard edge-cone $n$-sphere with an edge-cone singularity of cone angle $\alpha = 2\pi a$ along a great circle $\s^{n-2}$ for some $a \geq 2$. 
Then, there is no Yamabe minimizer on $(\s^n, h_a)$.  
\end{theorem}

\begin{proof}
We follow the lines of \cite{Au76}. Consider $\s^n_a = (\s^n, h_a)$ and $\s^n_b = (\s^n, h_b)$ for $b=\frac{a}{2}$. 
Then $(\s^n, h_a)$ is a double branched cover of $(\s^n, h_b)$ with branched set $\Sigma = \mathbb{S}^{n-2}$. 
Denote by $\mathcal{P}: \s^n_a \rightarrow \s^n_b$ the covering map. 
Suppose that a Yamabe minimizer $u > 0$ with $||u||_{\p}=1$ does exist on $(\s^n, h_a)$. 
Proposition \ref{reg} clearly applies to $(\s^n, h_a)$, 
thus we know that $u$ is H\"older continuous: it belongs to $C^{0,\nu}(\s^n_a)$ for $\nu=1/a$. 
Moreover, by Proposition \ref{reg} (ii) we control by $\eps^{\nu-1}$ the norm of its gradient away from an $\eps$-tubular neighbourhood of the singular set $\Sigma$. 
Starting from $u$, we are going to construct a function $v_0$ on $\s^n_b$ with the same regularity as $u$, which satisfies $Q_b(v_0)<Y(\s^n, [h_a]).$

Let $\mathcal{G} = \{ id, \gamma \}$ be the deck transformation group of the double branched covering $\mathcal{P}: \s^n_a \rightarrow \s^n_b$. 
Consider the average 
$$ 
v := u + u\circ \gamma\quad {\rm on}\ \ \s^n_a, 
$$ 
and define for $p > 0$
$$ 
v^{\langle p \rangle} := u^p + (u\circ \gamma)^p\quad {\rm on}\ \ \s^n_a. 
$$ 
Note that $\gamma$ is an isometry of $(\s^n, h_a)$ as well as $id$. 
As a consequence, $v$, $v^{\langle p \rangle}$ and their gradients have the same regularity as $u$ and $|du|$. 
We also observe that, since the $\gamma$ is an isometry, we have 
\begin{equation}
\label{A}
\int_{\s^n} v^{\langle \frac{2n}{n-2} \rangle}d\mu_a= 2\int_{\s^n}u^{\frac{2n}{n-2}}d\mu_a. 
\end{equation}
Notice that on the regular set $\Omega$ we have 
\begin{equation}
\label{B}
L_a v= L_a(u + u\circ \gamma)=Y_a(u^{\frac{n+2}{n-2}} + (u\circ \gamma)^{\frac{n+2}{n-2}}) = Y_av^{\langle \frac{n+2}{n-2} \rangle}. 
\end{equation}

Consider $v_0 \in C^{0,\nu}(\s^n_b)$ the function on $(\s^n, h_b)$ whose lift on $(\s^n, h_a)$ is $v$. 
We are going to show that $Q_{b}(v_0) < Y_a$. 
We first compute $Q_{b}(v_0)$,
$$Q_{b}(v_0)=\frac{\displaystyle \frac{1}{2}\int_{\s^n}(a_n|d\mu|^2 + R_av^2)d\mu_a}{2^{-\frac{n-2}{n}}||v||_{\p}^2}=2^{-\frac{2}{n}}\frac{E_a(v)}{||v||_{\p}^2}.$$
Thanks to Proposition \ref{reg} and in particular to \eqref{confIPP}, we know that 
$$E_a(v)=\int_{\s^n}vL_avd\mu_a,$$
and by using \eqref{B} we obtain 
\begin{equation*}
E_a(v)=Y_a\int_{\s^n} v\cdot v^{\langle \frac{n+2}{n-2} \rangle} d\mu_a. 
\end{equation*}
By H\"older's inequality, we get
$$E_a(v) \leq Y_a \int_{\s^n} v\left( u^{\frac{n}{n-2}} + (u\circ \gamma)^{\frac{n}{n-2}} \right)^{\frac{n-2}{n}}
\left( u^{\frac{2n}{n-2}} + (u\circ \gamma)^{\frac{2n}{n-2}} \right)^{\frac{2}{n}} d\mu_a$$
We now use that, since $u, u\circ \gamma > 0$, the following strict inequality holds 
$$ \left( u^{\frac{n}{n-2}} + (u\circ \gamma)^{\frac{n}{n-2}} \right)^{\frac{n-2}{n}} < u + u\circ \gamma\quad {\rm at\ each\ point\ of}\ \s^n_a, $$
and then obtain
$$E_a(v) < Y_a  \int_{\s^n} v^2 \left( u^{\frac{2n}{n-2}} + (u\circ \gamma)^{\frac{2n}{n-2}} \right)^{\frac{2}{n}} d\mu_a. $$
Finally, thanks to H\"older's inequality and \eqref{A}, we get 
$$\int_{\s^n} v^2 \left( u^{\frac{2n}{n-2}} + (u\circ \gamma)^{\frac{2n}{n-2}} \right)^{\frac{2}{n}} d\mu_a 
\leq ||v||^2_{\p} \left(\int_{\s^n}v^{\langle \p \rangle}\right)^{\frac 2n}
= ||v||_{\p}^2 2^{\frac{2}{n}}.$$
By collecting this information, we have proven that 
$$Q_b(v_0)< Y_a.$$
In particular: 
\begin{equation*}
\label{strict}
Y(\s^n, [h_b]) < Y(\s^n, [h_a]).
\end{equation*}

But since $a \geq 2$, both $a$ and $b$ are greater than or equal to one, therefore \cite{M17} ensures that 
both $Y(\s^n, [h_b])$ and $Y(\s^n, [h_a])$ coincide with ${\bf Y}_n$. 
This contradicts the previous strict inequality, thus a Yamabe minimizer cannot exist on $(\s^n, h_a)$. 
\end{proof}

\begin{rem}
We point out that an argument similar to the one of Viaclovsky \cite{V10} can hardly be applied in our setting. His proof is also by contradiction and goes as follows. Consider the orbifold conformal compactification $(\hat{X}, \hat{g})$ of $(X^4_n, g)$ a hyperk\"ahler ALE metric with group $G$ of order $n >1$ at infinity, and assume that there exists a Yamabe minimizer $\tilde{g}$ in the conformal class of $\hat{g}$. The metric $\hat g$ is Einstein, and Viaclovsky shows a generalization of a result by Obata for Einstein manifolds: if $\hat{g}$ is an Einstein metric, then any other constant scalar curvature metric $\tilde{g}=\varphi^{-2}g$ in its conformal class is an Einstein metric as well. This implies that on $(X^4_n,g)$ there exists a concircular scalar field: a rigidity result by Tashiro \cite{T65} for manifolds carrying such fields and the classification of hyperk\"ahler ALE 4-manifolds due to Kronheimer \cite{Kr89} lead then to a contradiction. 

Viaclovsky is able to use the result on the conformal class of an Einstein metric because in his case the conformal factor and its gradient are appropriately controlled in a neighbourhood of the orbifold singularity.


Now, Theorem 4.3 in \cite{M18} provides the analog of Obata's result for a compact stratified space endowed with an Einstein iterated edge metric with angles \emph{smaller} than $2\pi$ along the codimension 2 singular set. This strongly depends on the conformal factor being Lipschitz, with both bounded gradient and Laplacian. As we have seen in Section 2, when the cone angle along the codimension 2 singularity is greater than $2\pi$, the sharp regularity for a Yamabe minimizer is H\"older continuity, and therefore the proof Obata's result for stratified spaces does not hold. 
Indeed, for any constant scalar curvature metric $g = u^{4/(n-2)}\,h_a \in [h_a]$ on $\s^n$, 
$u$ has the following expansion in a standard transversal polar coordinate system $(\rho, x, \theta) \in (0, \frac{\pi}{2}) \times \s^{n-2} \times \s^1$ 
near the singular stratum $\s^{n-2}$: 
$$ 
u(\rho, x, \theta) = \phi_0(x, \theta) + \phi_1(x, \theta)\,\rho^{\frac{1}{a}} + \cdots, 
$$ 
where $\phi_0, \phi_1 \in C^{\infty}(\s^{n-2} \times \s^1)$, $L^{-1} \leq \phi_0 \leq L$ (for some $L \geq 1$) on $\s^{n-2} \times \s^1$, 
and moreover 
$$ 
\frac{\partial u}{\partial \rho} = O(\rho^{\frac{1}{a} - 1}). 
$$ 
Hence, when $a > 1$, $u$ is only $C^{0, \frac{1}{a}}$-H\"older continuous generally (see \cite[Section\,3]{ACM14} for details).  

\end{rem}

\section{Sobolev inequalities in a singular $\R^n$}
Thanks to the conformal equivalence between $(\s^n, h_a)$ and $(\R^n, g_a)$, 
we know that a Sobolev inequality holds on $(\R^n, g_a)$. 
Moreover, we know that the optimal constant of such Sobolev inequality is given by the Yamabe constant $Y_a$ of the standard edge-cone sphere $(\s^n, h_a)$. 
More precisely, if we denote by $\mu_a$ the measure associated to $g_a$, for any Sobolev function with compact support $u \in W^{1,2}_0(\R^n,\mu_a)$, 
we have 
\begin{equation}
\label{Sob}
Y_a \left(\int_{\R^n}u^{\p}d\mu_{a}\right)^{\frac{n-2}{n}}\leq \int_{\R^n} |du|^2d\mu_a.
\end{equation}
Note that $d\mu_a=ad\mu_1$, where $\mu_1$ is the usual Lebesgue measure on $\R^n$. 
The previous inequality can be extended to functions $u \in W^{1,2}_{loc}(\R^n, \mu_a)$. 

When considering $\R^n$ endowed with the Euclidean metric $g_1 = g_{\mathbb{E}}$, 
we know that there exist extremal functions attaining the equality in the Sobolev inequality \eqref{Sob}. 
They are the functions $u_{\lambda,v}$ defined in Section 3, 
by composing the conformal factor of $\varphi^*h_1$ with a dilation $\delta_{\lambda}$ and a translation $\tau_v$. 
We also know that for $a \in (0,1]$
$$\varphi^*h_a=4u_{\lambda,v}^{\frac{4}{n-2}}g_a$$
are Yamabe metrics. 
As a consequence, for $a \in (0,1]$ and for any $\lambda \in \R$, vector $v \in \R^n$, 
the functions $u_{\lambda,v}$ attain the equality in the Sobolev inequality \eqref{Sob}. 
As for $a > 1$, observe the following 
\begin{equation*}
\begin{split}
{\bf Y}_n \left(\int_{\R^n}|u_{\lambda,v}|^{\p}d\mu_a\right)^{\frac{n-2}{n}} & 
= {\bf Y}_n a^{\frac{n-2}{n}}\left(\int_{\R^n} |u_{\lambda,v}|^{\p}d\mu_1  \right)^{\frac{n-2}{n}} \\
& = a^{\frac{n-2}{n}}\left(\int_{\R^n}|du_{\lambda,v} |^2 d\mu_1 \right) \\
& = a^{\frac{n-2}{n}-1} \left(\int_{\R^n}|du_{\lambda,v}|^2 d\mu_a \right).
\end{split}
\end{equation*}
This implies  
\begin{equation*}
Y_a\left( \int_{\R^n}u^{\p}d\mu_{a}\right)^{\frac{n-2}{n}} < a^{\frac{2}{n}} {\bf Y}_n \left(\int_{\R^n}|u_{\lambda,v}|^{\p}d\mu_a\right)^{\frac{n-2}{n}} = \int_{\R^n}|du_{\lambda,v}|^2 d\mu_a. 
\end{equation*}
This means that the functions $u_{\lambda,v}$ never attain the equality in \eqref{Sob} when $a > 1$. 
This corresponds to the fact that $h_a$ is not a Yamabe metric for $a >1$. 
We also know that for $a \geq 2$, the conformal class of $h_a$ does not contain any Yamabe metric. 
As a consequence, the equality in \eqref{Sob} cannot be attained when $a\geq 2$. 
We can summarize this in the following: 

\begin{prop}
Consider $\R^n$ endowed with the metric $g_a = g_{\mathbb{E}^{n-2}} + dr^2 + a^2r^2d\theta^2$. 
\begin{itemize}
\item[(a)] If $a \leq 1$, for any $u \in W^{1,2}_{loc}(\R^n)$ we have
\begin{equation*}
 a^{\frac{2}{n}}{\bf Y}_n  \left(\int_{\R^n}|u|^{\p}d\mu_a\right)^{\frac{n-2}{n}} \leq \int_{\R^n}|du|^2 d\mu_a.
\end{equation*} 
Moreover, the functions $ u_{\lambda,\mu}$ are extremal for the previous inequality. 
\item[(b)] If $a >1$, for any $u \in W^{1,2}_{loc}(\R^n)$ we have: 
\begin{equation*}
{\bf Y}_n\left(\int_{\R^n}|u|^{\p}d\mu_a\right)^{\frac{n-2}{n}} \leq \int_{\R^n}|du|^2 d\mu_a.
\end{equation*}
Moreover, if $a \geq 2$, extremal functions for the latter inequality do not exist. 
\end{itemize}
\end{prop}
In the second case, an interesting question would be to study the existence in $(\R^n,g_a)$, $a >1$, of not necessarily positive solutions to the Yamabe equation (see for example \cite{DMPP}).


\vspace{10mm} 

\bibliographystyle{amsbook}

\vspace{10mm}

\end{document}